\documentclass[a4paper,11pt]{amsart}
\usepackage[english]{babel}
\usepackage{amsmath,amsthm}
\usepackage{amsfonts}
\usepackage{enumerate}
\usepackage{graphicx}
\usepackage{color}
\usepackage[all]{xy}
\usepackage{thmtools, thm-restate}
\usepackage{overpic}
\usepackage{amssymb,tikz}
\usepackage{quiver}
\usepackage{enumerate}
\usepackage[dvipsnames]{xcolor}
\usepackage[pagebackref,hypertexnames=false,colorlinks=true,
            linkcolor=NavyBlue,
            citecolor=NavyBlue,
            urlcolor=NavyBlue]{hyperref} 
\usepackage{cleveref}
\usepackage{longtable}
\usepackage{array}
\newcolumntype{Y}{>{\centering\arraybackslash}p{1.2cm}}

\usepackage[top=.9in, bottom=.9in, left=1in, right=1in]{geometry}

\newtheorem{thm}{Theorem}%[section]
\newtheorem{cor}[thm]{Corollary}

\newtheorem{prop}[thm]{Proposition}

\newtheorem{claim}{Claim}

\theoremstyle{definition}

\theoremstyle{remark}

\numberwithin{equation}{section}
\newcommand{\wt}[1]{\widetilde{ #1}}

\newcommand{\spin}{\ifmmode{\rm Spin}\else{${\rm spin}$\ }\fi}
\newcommand{\spinc}{\ifmmode{{\rm Spin}^c}\else{${\rm spin}^c$}\fi}

\newcommand{\Z}{\mathbb{Z}}

\newcommand{\CP}{\mathbb{CP}}
\newcommand{\rk}{{\rm rk}\,}

\newcommand{\plumb}{\entrymodifiers={+[o][F-]} \xymatrix@C=10pt}
\newcommand{\el}{\ar@{-}[r]}
\newcommand{\ed}{\ar@{--}[r] }

\begin{document}

\title{Special alternating links of minimal unlinking number}%

% \author{Paolo Aceto}%
% \address {Université de Lille}
% \email{paoloaceto@gmail.com }

\author{Duncan McCoy}%
\address {Université du Québec à Montréal}
\email{mc\_coy.duncan@uqam.ca}

\author{JungHwan Park}%
\address {Korea Advanced Institute of Science and Technology}
\email{jungpark0817@kaist.ac.kr}
%\date{\today}%

\begin{abstract}
For any link in the 3-sphere, there is a natural lower bound for the unlinking number in terms of the classical signature. We prove that if this lower bound is sharp for a special alternating link $L$, then the unlinking number of $L$ is necessarily realized by crossing changes in any alternating diagram for $L$. As an application, we compute new values of the unknotting numbers for some special alternating knots with crossing number 11 and 12.
\end{abstract}
\maketitle

%\tableofcontents
% ----------------------------------------------------------------

\section{Introduction}
Given a link $L$ in the 3-sphere $S^3$, its unlinking number $u(L)$ is defined to be the minimal number of crossing changes required in any diagram of $L$ to obtain an unlink. Despite the simple definition, this invariant is often rather challenging to compute since one typically does not know which diagrams will realize the unlinking number. Indeed results which predict exactly where the unknotting or unlinking numbers will be realized are rare. For example, among the few general results available, it is known that alternating knots with unknotting number one can always be unknotted by a single crossing change in any alternating diagram\footnote{Recall that a reduced alternating diagram is \emph{minimal}, meaning that it realizes the crossing number.} \cite{McCoy2017alternating} and 2-bridge links with unlinking number one can always be unlinked with a single crossing change in any alternating diagram \cite{Kohn1991unlinking_2bridge} (see also \cite{Kanenobu-Murakami:1986, Gordon-Luecke:2006, Greene:2014}). However, higher unknotting numbers can not necessarily be computed using minimal diagrams, even for alternating knots \cite{Bleiler1984note, Bernhard1994unknotting, Brittenham2021BJ}. The recent striking examples of Brittenham and Hermiller suggest that the unknotting number is far more subtle than was previously imagined  \cite{Brittenham2025, Brittenham2026}.

For every oriented non-split $k$-component alternating link $L$ there is a lower bound for the unlinking number of the form 
\begin{equation}\label{eq:alt_sig_bound}
u(L)\geq \frac{|\sigma(L)|+k-1}{2},
\end{equation}
where $\sigma(L)$ denotes the classical signature.\footnote{Links in this paper are always assumed to be oriented.} The main result of this paper is to show that if $L$ is a special alternating link for which this lower bound is sharp, then the unlinking number is realized by crossing changes in any alternating diagram for $L$. Recall that a special alternating link is one admitting an alternating diagram in which one of the planar spanning surfaces is a Seifert surface for $L$. We obtain a similar result for the 4-ball crossing number $c_4(L)$ of a special alternating link. By definition, $c_4(L)$ is the minimal number of transverse double points among all properly immersed unions of disks in the $4$-ball $B^4$ with boundary $L$. %The classical signature of $L$ is denoted by $\sigma(L)$. 
 
\begin{thm}\label{thm:main}
    Let $L$ be an oriented non-split special alternating link with $k$ components. Then the following are equivalent:
    \begin{enumerate}[(i)]
        \item\label{it:unlinking} $u(L)=\frac{|\sigma(L)|+k-1}{2}$;
        \item\label{it:c4} $c_4(L)=\frac{|\sigma(L)|+k-1}{2}$;
        \item\label{it:crossing_changes} $L$ can be unlinked by $\frac{|\sigma(L)|+k-1}{2}$ crossing changes in any alternating or minimal diagram.
    \end{enumerate}
\end{thm}
Although every minimal diagram for a prime alternating link is alternating \cite{Kauffman1987, Murasugi1987Jones, Thistlethwaite1987}, non-prime alternating links admit non-alternating minimal diagrams. Hence there is no redundancy in the inclusion of both alternating and minimal diagrams in Theorem~\ref{thm:main}\eqref{it:crossing_changes}.

The most interesting applications of Theorem~\ref{thm:main} concern the unknotting number.
\begin{cor}\label{cor:unknotting}
    Let $K$ be a special alternating knot. Then $u(K)=|\sigma(K)|/2$ if and only if $K$ can be unknotted by $|\sigma(K)|/2$ crossing changes in any alternating or minimal diagram. \qed
\end{cor}

For a special alternating knot $K$, Corollary~\ref{cor:unknotting} gives an explicit, implementable procedure to decide whether $u(K)=|\sigma(K)|/2$. As described in Section~\ref{sec:calculations}, this allows us to compute $u(K)$ for many small examples in the knot tables, including several knots of crossing number $11$ and $12$ for which $u(K)$ was previously undetermined.

We remark that Theorem~\ref{thm:main} and Corollary~\ref{cor:unknotting} do not hold for all alternating knots and links. Indeed, there exist alternating knots $K$ with $u(K)=|\sigma(K)|/2$ such that this value cannot be realized by crossing changes in any alternating diagram \cite{Bleiler1984note, Bernhard1994unknotting}.

We note also the following corollary on the additivity of the unknotting number. An analogous result holds for links, but for simplicity we state it only for knots.
\begin{cor}\label{cor:additive}
Let $K$ be a special alternating knot with $K=K_1 \# K_2$ and $u(K)=|\sigma(K)|/2$. Then
\[\pushQED{\qed}
u(K_1 \# K_2)=u(K_1)+u(K_2). \qedhere
\]
\end{cor}
This corollary is obtained by applying Corollary~\ref{cor:unknotting} to a reduced alternating diagram for $K$, and recalling that such a diagram is necessarily a connected sum of diagrams for $K_1$ and $K_2$ \cite{Menasco1984incompressible}.

Finally, we note that Theorem~\ref{thm:main} can be extended to a statement about split links. Since the unlinking number and the $4$-ball crossing number are additive under split union, we have the following.

\begin{cor}\label{cor:split}
Let $L$ be the split union of oriented non-split special alternating links $L_1,\dots,L_\ell$ and let
\[
m=\sum_{i=1}^\ell \frac{|\sigma(L_i)|+k_i-1}{2},
\]
where $k_i$ denotes the number of components of $L_i$. Then the following are equivalent:
\begin{enumerate}[(i)]
\item\label{it:unlinking_split} $u(L)=m$;
\item\label{it:c4_split} $c_4(L)=m$;
\item\label{it:crossing_changes_split} $L$ can be unlinked by $m$ crossing changes in any alternating or minimal diagram. \qed
\end{enumerate} 
\end{cor}

% Lastly, note that if two knots $K_1$ and $K_2$ satisfy $u(K_i)=-\sigma(K_i)/2$ for $i=1,2$, then the classical signature bound, together with additivity of the signature and subadditivity of the unknotting number, implies that the unknotting number is additive for $K_1$ and $K_2$. Corollary~\ref{cor:unknotting} shows that the same conclusion still holds even when one of the knots satisfies $u(K_i)=(-\sigma(K_i)/2)+1$.

% \begin{cor}\label{cor:additive}
% Let  $K_1$ and $K_2$ be special alternating knots. If
% \[
% u(K_1)=-\frac{\sigma(K_1)}{2}
% \qquad\text{and}\qquad
% u(K_2)=-\frac{\sigma(K_2)}{2}+1,
% \]
% then
% \[\pushQED{\qed}
% u(K_1 \# K_2)=u(K_1)+u(K_2). \qedhere
% \]
% \end{cor}

 \subsection*{Structure of proof}
For any oriented $k$-component link $L$, there is a lower bound on $c_4(L)$ in terms of classical invariants:
\begin{equation}\label{eq:signature_bound}
    c_4(L)\geq \frac{|\sigma(L)|-\eta(L)+k-1}{2},
\end{equation}
where $\eta(L)$ denotes the nullity of $L$ \cite[Lemma 2.2]{NagelOwens}. Since the nullity $\eta(L)$ of a non-split alternating link is always zero, the implications \eqref{it:crossing_changes}$\Rightarrow$\eqref{it:unlinking} and \eqref{it:unlinking}$\Rightarrow$ \eqref{it:c4} in Theorem~\ref{thm:main} follow immediately from \eqref{eq:signature_bound} and the fact that $u(L)\geq c_4(L)$. 

Consequently, the content of Theorem~\ref{thm:main} is the implication \eqref{it:c4}$\Rightarrow$\eqref{it:crossing_changes}. This implication is established by using techniques from smooth 4-manifold topology. In particular, for alternating links we derive a lattice-theoretic obstruction to equality in \eqref{eq:signature_bound} using Donaldson's diagonalization theorem \cite{Donaldson:1987-1} (see Theorem~\ref{thm:lattice_obstruction} for the precise statement). Although similar obstructions have appeared previously in the literature \cite{NagelOwens, Feller2016twobridge, Owens2010slicing}, we obtain a particularly refined obstruction by exploiting the fact that one of the $4$-manifolds arising in the proof is spin (see Proposition~\ref{prop:spinNagelOwens}). The proof of Theorem~\ref{thm:main} is completed by analyzing this lattice-theoretic obstruction in the case of special alternating links. This analysis is carried out in Section~\ref{sec:proof}.

\subsection*{Acknowledgements}
DM is partially supported by NSERC grant RGPIN-2020-05491 and a Canada Research Chair. JP is partially supported by Samsung Science and Technology Foundation (SSTF-BA2102-02) and the NRF grant RS-2025-00542968. Both authors are grateful to Paolo Aceto for many helpful discussions, without which this paper would never been conceived.

\section{Obstructing minimal unlinking numbers}
The aim of this section is to derive the obstruction to an alternating link realizing equality in \eqref{eq:signature_bound}.
The following construction is a specialization of that of Nagel--Owens \cite[Proposition~2.3]{NagelOwens}, which was in turn inspired by Cochran--Lickorish \cite{CochranLickorish}. The key observation is that, in our setup, the construction yields a spin $4$-manifold. Throughout the article, we work in the smooth category.

\begin{prop}\label{prop:spinNagelOwens}
Let $L$ be an oriented $k$-component link in $S^3$ with
\begin{equation}\label{eq:c4sig_condition}
    c_4(L) = \frac{-\sigma(L)-\eta(L) +k-1}{2}.
\end{equation}
Then the double branched cover $\Sigma(L)$ bounds a
smooth, spin, negative-definite 4-manifold $W$ with
$b_2(W) = -\sigma(L)$.
Moreover $H_2(W; \Z)$ contains $c_4(L)$ pairwise disjoint spherical
classes of self-intersection $-2$.
\end{prop}
\begin{proof} Let $p=c_4(L)$. 
By \cite[Lemma~2.2]{NagelOwens}, we know that if \eqref{eq:c4sig_condition} holds, then $L$ bounds a properly immersed union of $k$ disks in $B^4$ with $p$ double points which are necessarily all positive. By blowing-up at each of these double points, we obtain $L$ as the boundary of a properly embedded collection of $k$ disjoint disks
\[
\Delta \hookrightarrow X:=B^4\#\left( \#_p \overline{\CP}^2\right),
\]
where $\Delta$ represents the homology class of two-times the spherical generator in each $\overline{\CP}^2$ summand. We take $W$ to be the double cover of $X$ branched along $\Delta$. Since $\Delta$ is oriented and the Poincaré dual of the homology class $[\Delta]/2$ reduces mod 2 to the second Stiefel-Whitney class of $X$, the manifold $W$ is spin \cite{NagamiSpincCovers}. Since this construction is precisely the one used in \cite[Proposition~2.3]{NagelOwens}, the remaining properties of $W$ follow from the arguments therein.
\end{proof}
Let $D$ be a reduced alternating diagram for a non-split link $L$. The diagram divides the plane into connected regions which may be coloured black and white in a chessboard manner. There are two possible choices of colouring, each giving an incidence number, $\mu(c)\in \{\pm 1\}$, at each crossing $c$ of $D$, as shown in Figure~\ref{fig:incidencenumber}.
 \begin{figure}[ht]
   \centering
%   %\def\svgwidth{\columnwidth}
   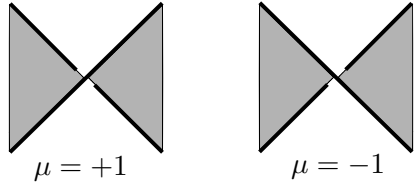
  \caption{The incidence number of a crossing.}
  \label{fig:incidencenumber}
 \end{figure}
 Since $D$ is alternating, we may fix the colouring so that $\mu(c)=-1$ for all crossings. We define the positive-definite Goeritz lattice of $D$, denoted $\Lambda_D$, as follows \cite[Chapter~9]{lickorish1997introduction}. Let $v_0, \dots, v_r$ be the white regions of the diagram $D$. Consider the free abelian group generated by $v_0, \dots, v_r$ with the symmetric bilinear pairing such that
 \[
 v_i \cdot v_j=\begin{cases}
        c_{ii}& \text{if $i= j$ and there are $c_{ii}$ crossings around $v_i$}\\
     -c_{ij} &\text{if $i\neq j$ and there are $c_{ij}$ crossings between $v_i$ and $v_j$}.
 \end{cases}
 \]
 This pairing is degenerate, but it descends to a positive-definite pairing on the quotient\footnote{This uses the fact that $D$ is non-split.}
\[
\Lambda_D=\frac{\Z\langle v_0, \dots, v_r \rangle}{\langle v_0+\dots + v_r\rangle}.
\]
We refer to $\Lambda_D$, equipped with this symmetric form, as the positive-definite Goeritz lattice of $D$. The lattice $\Lambda_D$ could be described more concisely, but less explicitly, as the Gordon-Litherland form on $H_1(S_+;\Z)$, where $S_+$ is the spanning surface obtained from the black regions in the checkerboard colouring \cite{GordonLitherland}. 

For simplicity, we refer to the standard positive-definite lattice $(\Z^n,\langle 1\rangle^n)$ as $\Z^n$. Our lattice-theoretic obstruction is the following.
\begin{thm}\label{thm:lattice_obstruction}
Let $L$ be a non-split $k$-component alternating link with $\sigma(L)\leq 0$. Let $D$ be a reduced alternating diagram for $L$ with positive-definite Goeritz lattice $\Lambda_D$. Let 
\[p=\frac{-\sigma(L)+k-1}{2} \quad\text{and}\quad n=\rk \Lambda_D-\sigma(L).\] If $c_4(L)=p$, then there is a lattice embedding of $\Lambda_D$ into $\Z^n$ such that the following conditions are satisfied:
\begin{enumerate}[(i)]
    \item\label{it:all_coords} for all unit vectors $e\in \Z^n$, there exists $v$ in the image of $\Lambda_D$ such that $v\cdot e\neq 0$ and
    \item\label{it:u_disjoint_supports} there exists an orthonormal basis $e_1,\dots, e_n$ for $\Z^n$ such that for all $i=1, \dots, p$, every $v$ in the image of $\Lambda_D$ satisfies $v\cdot e_{i}=v\cdot e_{i+p}$.
\end{enumerate}
\end{thm}
\begin{proof}
    Since $L$ is a non-split alternating link, we have $\eta(L)=0$.
There is a compact $4$-manifold $X_D$ with boundary $\partial X_D \cong \Sigma(L)$ and intersection form given by the positive-definite Goeritz form associated to $D$ \cite[Section~3]{ozsvath2005heegaard}. Moreover, since $X_D$ can be constructed without $1$-handles, we have $H_1(X_D;\Z)=0$.

Since we are assuming that
\[
c_4(L)=\frac{-\sigma(L)+k-1}{2},
\]
Proposition~\ref{prop:spinNagelOwens} yields a negative-definite spin $4$-manifold $W$ with $p$ pairwise disjoint spherical classes of self-intersection $-2$ such that $\partial W \cong \Sigma(L)$.

This allows us to form the closed smooth manifold
\[
Z=X_D\cup_{\Sigma(L)} -W
\]
with $b_2(Z)=\rk(\Lambda_D)-\sigma(L)=n$, where $-W$ denotes the manifold $W$ with reversed orientation.
Since $Z$ is smooth and positive-definite, its intersection form is diagonalizable \cite{Donaldson:1987-1}. Therefore, the inclusions
\[
X_D,\,-W \hookrightarrow Z
\]
induce a map of lattices
\[
\iota:\Lambda_D \oplus Q_{-W} \hookrightarrow \Z^{n},
\]
where $Q_{-W}$ is the integral lattice induced by the intersection form of $-W$.
Moreover, since $b_1(X_D)=0$, the lattice $Q_{-W}$ is isomorphic to the orthogonal complement of $\iota(\Lambda_D)$ in $\Z^{n}$ \cite[Lemma~2.4]{Aceto-McCoy-Park:2022}. Since $W$ is spin, its intersection form is even. Thus the orthogonal complement of $\iota(\Lambda_D)$ is even and contains $p$ pairwise orthogonal vectors of norm $2$. In particular, the orthogonal complement to $\iota(\Lambda_D)$ does not contain any vectors of norm one. Thus if $e$ is a unit vector in $\Z^{n}$, there is an element $v$ in $\iota(\Lambda_D)$ with $e\cdot v\neq 0$. This proves \eqref{it:all_coords}.

    It remains to establish \eqref{it:u_disjoint_supports}. First, let $u$ and $u'$ be a pair of orthogonal norm two vectors in $\Z^n$. Suppose that there exists a unit vector $e$ such that $u\cdot e\neq 0$ and $u'\cdot e\neq 0$. Up to replacing $u$ with $-u$ or $u'$ with $-u'$, we see that $u$ and $u'$ must take the form $u=e+f$ and $u'=e-f$, where $f$ is another unit vector orthogonal to $e$. Let $v$ be any vector that is orthogonal to both $u$ and $u'$. Since $u+u'=2e$, it follows that $v\cdot e=0$.
    
    Let $u_1,\dots,u_p$ be pairwise orthogonal vectors of norm two in the orthogonal complement of $\iota(\Lambda_D)$. Then for each $v$ in the image of $\Lambda_D$ we have $v\cdot u_i=0$. By \eqref{it:all_coords} and the preceding observation, given any distinct pair $u_i$ and $u_j$ there is no unit vector $e\in \Z^n$ such that $u_i\cdot e\neq 0$ and $u_j\cdot e\neq 0$. Thus we may choose an orthonormal basis $e_1,\dots,e_n$ for $\Z^n$ such that each $u_i$ has the form $u_i=e_i-e_{i+p}$. It follows that for any $v$ in the image of $\Lambda_D$,
\[
0=v\cdot u_i = v\cdot e_i - v\cdot e_{i+p},
\]
for all $i=1,\dots,p$. This establishes \eqref{it:u_disjoint_supports}.
\end{proof}
    
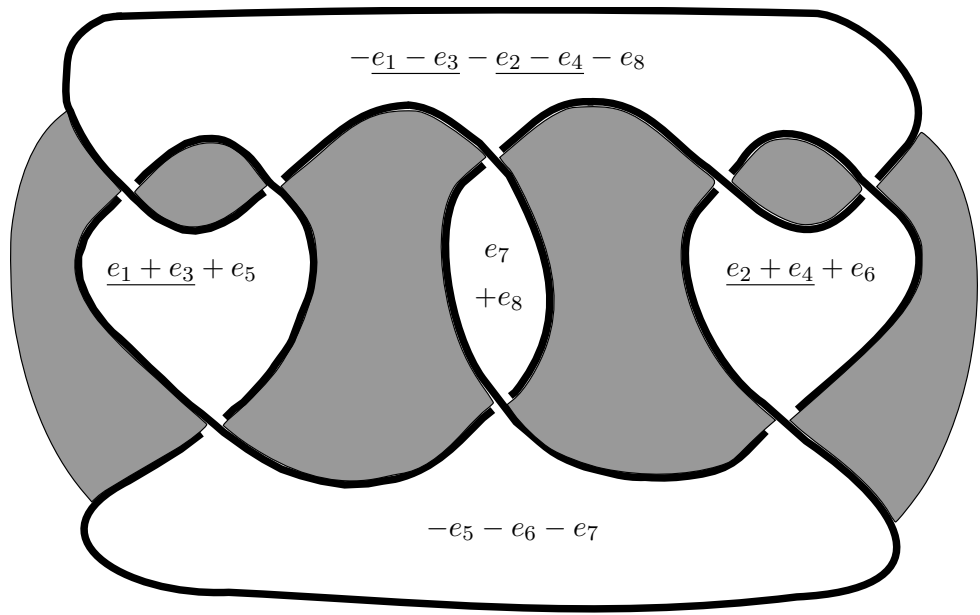
\begin{figure}
    \begin{overpic}[width=0.8\textwidth]{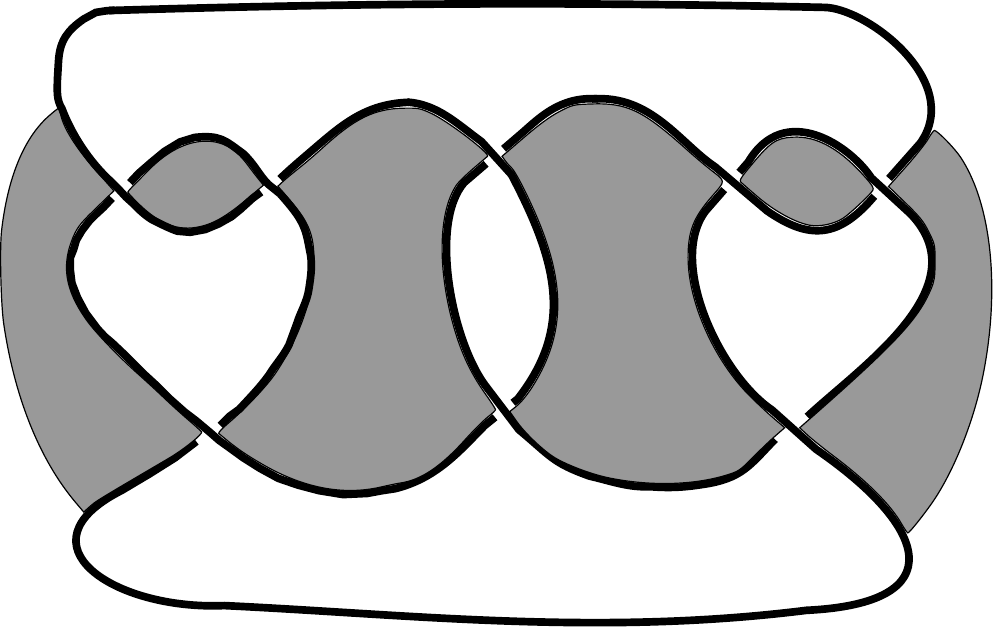}
      \put (10,35) {$\underline{e_1+e_3}+e_5$}
      \put (74,35) {$\underline{e_2+e_4}+e_6$}
      \put (49,37) {$e_7$}
      \put (48,32) {$+e_8$}
      \put (35,57) {$-\underline{e_1-e_3}-\underline{e_2-e_4}-e_8$}
      \put (43,8) {$-e_5-e_6-e_7$}
    \end{overpic}
  \caption{A diagram of the special alternating knot $8_{15}$, which has signature $-4$. The labellings in the white regions correspond to an embedding of the Goeritz form into $\Z^8$ satisfying the conclusions of Theorem~\ref{thm:lattice_obstruction}. Since $p=2$, we have $v\cdot e_1 = v\cdot e_3$ and $v\cdot e_2 = v\cdot e_4$ for each $v$ in the image of $\Lambda_D$. Note that $8_{15}$ can be unknotted by changing crossings between the regions labelled by the vectors $e_1$ and $e_2$.}
  \label{fig:8_15}
 \end{figure}

\section{Proof of Theorem~\ref{thm:main}}\label{sec:proof}
Let $L$ be a non-split special alternating link with $k$ components. As mentioned in the introduction, for any oriented link $L$ we have an inequality of the form \cite[Lemma~2.2]{NagelOwens}:
\begin{equation}\label{eq:u_c_4_sig_inequalities}
    u(L)\geq c_4(L)\geq \frac{|\sigma(L)|-\eta(L)+k-1}{2}.
\end{equation}
However, for a non-split alternating link we have $\eta(L)=0$. Thus \eqref{eq:u_c_4_sig_inequalities} yields the implications \eqref{it:crossing_changes}$\Rightarrow$\eqref{it:unlinking}$\Rightarrow$\eqref{it:c4}.

We complete the proof by establishing \eqref{it:c4}$\Rightarrow$\eqref{it:crossing_changes}. Let $p$ denote the quantity
\[
p=\frac{|\sigma(L)|+k-1}{2},
\]
and suppose that $c_4(L)=p$.

Let $D$ be a reduced alternating diagram for $L$ with $n$ crossings. Since $L$ is a special alternating link, the diagram $D$ contains only positive or negative crossings. Therefore, by mirroring if necessary, we may assume that $D$ is a positive diagram. Furthermore, by flyping if necessary, we may assume that $D$ is twist-reduced \cite[Section~1.1]{LackenbyPurcell}. \emph{Twist-reduced} means that for any pair of regions in a checkerboard colouring all crossings between these two regions lie in a single twist region.  

Colouring $D$ in a checkerboard fashion so that every crossing has incidence $\mu(c)=-1$, we obtain two connected spanning surfaces $S_-$ and $S_+$ for $L$, where $S_-$ is the white spanning surface and $S_+$ is the black spanning surface. Since every crossing in $D$ is positive, the white surface $S_-$ is a Seifert surface for $L$.
\begin{figure}
   \centering
%   %\def\svgwidth{\columnwidth}
   \includegraphics[width=0.15\textwidth]{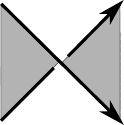}
  \caption{A shaded positive crossing}
  \label{fig:positive_crossing}
 \end{figure}
 Since the Gordon-Litherland pairing on $S_-$ is negative-definite, the signature of $L$ may be calculated as $\sigma(L)=-b_1(S_-)$. In particular, the Euler characteristic of $S_-$ satisfies
\begin{equation}\label{eq:chiS-}
\chi(S_-)=1+\sigma(L)=k-2p.
\end{equation}
Let $\Lambda_D$ be the positive-definite Goeritz form associated to $D$. We will use $v_0,\dots, v_r$ to denote the white regions of the checkerboard colouring and will identify them with their images in $\Lambda_D$. Since the Goeritz form $\Lambda_D$ is precisely the Gordon-Litherland pairing on $S_+$, we have $\rk \Lambda_D= b_1(S_+)=r$. Consequently,
\[
\rk \Lambda_D-\sigma(L)=b_1(S_+)+b_1(S_-)=n,
\]
where $n$ is the number of crossings of $D$.
Consider the embedding of $\Lambda_D$ into $\Z^{n}$ given by Theorem~\ref{thm:lattice_obstruction} and take $e_1, \dots, e_n$ to be an orthonormal basis satisfying Theorem~\ref{thm:lattice_obstruction}\eqref{it:all_coords} and \eqref{it:u_disjoint_supports}. To simplify notation,  we will identify $\Lambda_D$ with its image in $\Z^n$ under this embedding. Intuitively we may view each of the white regions $v_0,\dots, v_r$ of $D$ as being labelled by vectors in $\Z^n$ (see Figure~\ref{fig:8_15}, for example). The first step in the proof is to show that this labelling must be very simple.

\begin{claim}\label{claim1}
    For any basis vector $e_i\in \Z^n$, there are two regions $v_\alpha, v_\beta$ such that 
    \[v_\alpha \cdot e_i = +1,\quad v_\beta \cdot e_i=-1\] and for all other $j\neq \alpha,\beta$ we have $v_j\cdot e_i=0$.
\end{claim}
\begin{proof}[Proof of Claim]
By Theorem~\ref{thm:lattice_obstruction}\eqref{it:all_coords}, the lattice $\Lambda_D$ uses every coordinate, and therefore
\begin{equation}\label{eq:1}
\sum_{j=0}^r \bigl|v_j\cdot e_i\bigr|>0
\end{equation}
for all $i=1,\dots,n$. On the other hand, the fact that $\sum_{j=0}^r v_j=0$ implies that the left-hand side of \eqref{eq:1} is even. Thus,
\begin{equation}\label{eq:2}
\sum_{j=0}^r \bigl|v_j\cdot e_i\bigr|\ge 2
\end{equation}
for all $i=1,\dots,n$. The number of crossings of $D$ can be computed via
\begin{equation}\label{eq:3}
2n=\sum_{j=0}^r v_j\cdot v_j
=\sum_{j=0}^r\sum_{i=1}^n \bigl|v_j\cdot e_i\bigr|^2
\ge \sum_{i=1}^n\sum_{j=0}^r \bigl|v_j\cdot e_i\bigr|.
\end{equation}
If we sum the inequalities in \eqref{eq:2} over all $i$ and compare with \eqref{eq:3}, then we obtain
\begin{equation}\label{eq:4}
\sum_{j=0}^r \bigl|v_j\cdot e_i\bigr|=2
\end{equation}
for all $i=1,\dots,n$. Invoking once again that $\sum_{j=0}^r v_j=0$, the claim follows from \eqref{eq:4}.
\end{proof}
Crucially, Claim~\ref{claim1} implies that for any distinct pair of regions $v_\alpha$ and $v_\beta$, the number of crossings between $v_\alpha$ and $v_\beta$ is given by
\[
-v_\alpha \cdot v_\beta
=
\bigl|\{\,e_i \mid v_\alpha \cdot e_i \neq 0 \text{ and } v_\beta \cdot e_i \neq 0\,\}\bigr|.
\]\vspace{-1.2\baselineskip}
\begin{claim}\label{claim2}
The diagram $D$ contains at least $p$ disjoint clasps between white regions.
\end{claim}
\begin{proof}[Proof of Claim]
For a basis vector $e_i$, we refer to the pair of white regions $v_\alpha$ and $v_\beta$ satisfying $v_\alpha\cdot e_i\neq 0$ and $v_\beta\cdot e_i\neq 0$ as the regions \emph{marked} by $e_i$. In particular, the number of crossings between two distinct regions $v_\alpha$ and $v_\beta$ is exactly the number of basis vectors that mark the pair $v_\alpha$ and $v_\beta$. By Theorem~\ref{thm:lattice_obstruction}\eqref{it:u_disjoint_supports}, if the pair $v_\alpha$, $v_\beta$ is marked by $e_i$ for some $1\le i\le p$, then it is also marked by $e_{i+p}$. Hence, if a pair of regions $v_\alpha$, $v_\beta$ is marked by $k\ge 1$ of the vectors $e_1,\dots,e_p$, then there are at least $2k$ crossings between $v_\alpha$ and $v_\beta$. Since $D$ is twist reduced, these $2k$ crossings lie in a single twist region. Thus we may select $k$ disjoint clasps from the twist region between $v_\alpha$ and $v_\beta$. Considering all pairs of regions marked by at least one of $e_1,\dots,e_p$, we obtain the required $p$ disjoint clasps.
\end{proof}

Let $L'$ be the link obtained by changing one crossing from each of the $p$ disjoint clasps between white regions in $D$.
\begin{claim}
$L'$ is the unlink on $k$ components.
\end{claim}
\begin{proof}[Proof of Claim]
If we change a crossing in a clasp between white regions of $D$, then the resulting link is isotopic, by a Reidemeister~II move, to a link $\wt{L}$ which bounds a surface $\wt{S}$ obtained from $S_-$ by deleting two bands (one for each crossing in the clasp). See Figure~\ref{fig:band_removal}. Thus, if we perform $p$ crossing changes, each in a disjoint clasp, then the resulting link $L'$ bounds a surface $S'$ obtained from $S_-$ by deleting $2p$ bands. Using \eqref{eq:chiS-} to compute the Euler characteristic of $S'$, we obtain
\[
\chi(S')=\chi(S_-)+2p=k.
\]
Since $S'$ has $k$ boundary components and no closed components, it follows that $S'$ is a union of $k$ disks. In particular, $L'$ is the $k$-component unlink.
\end{proof}

\begin{figure}
   \centering
%   %\def\svgwidth{\columnwidth}
   \includegraphics[width=0.7\textwidth]{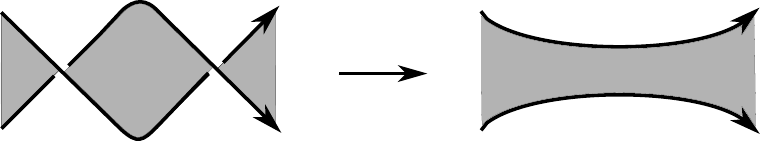}
  \caption{Changing a crossing in a clasp corresponds to the deletion of two bands from the white surface $S_-$.}
  \label{fig:band_removal}
 \end{figure}
 
Thus we have established that the alternating diagram $D$ can be unlinked by $p$ crossing changes. Since crossing changes commute with flypes, and any two reduced alternating diagrams are related by a sequence of flypes \cite{Menasco1993flyping}, it follows that any reduced alternating diagram for $L$ can be unlinked by $p$ crossing changes. Since every alternating diagram of a link can be converted to a reduced alternating diagram after removing nugatory crossings, it follows that every alternating diagram for $L$ can be unlinked by $p$ crossing changes.

Finally, we deduce that any minimal diagram for $L$ can be unlinked by changing $p$ crossings. Suppose that $L$ can be decomposed as a connected sum $L=L_1\# \dots \# L_\ell$, where each $L_i$ is prime. Since non-prime alternating links do not admit prime alternating diagrams \cite{Menasco1984incompressible}, it follows that each $L_i$ is itself alternating. Moreover, since any alternating diagram for $L$ can be unlinked by $p$ crossing changes, it follows that any reduced alternating diagram $D_i$ for $L_i$ can be unlinked by $p_i$ crossing changes\footnote{In fact, one can verify that $p_i=\frac{|\sigma(L_i)|+k_i-1}{2}$, where $L_i$ has $k_i$ components.}, where   
\[
p=p_1+\dots +p_\ell.
\]
Since every minimal diagram for $L$ is a connected sum of reduced alternating diagrams for the $L_i$ \cite{Murasugi1987Jones}, it follows that any minimal diagram for $L$ can be unknotted by $p$ crossing changes, as required.

This completes the proof of \eqref{it:c4}$\Rightarrow$\eqref{it:crossing_changes}.

\section{Applications to knots with small crossing number}\label{sec:calculations}
The unknotting number and the $4$-ball crossing number of special alternating knots with crossing number at most $9$ have all been determined and are recorded in \cite{knotinfo}.  For special alternating knots with crossing number $10$, the unknotting number is also known and recorded in \cite{knotinfo} and the $4$-ball crossing number is known to coincide with the unknotting number \cite{Owens2016Immersed}. With the exception of $9_{35}$ which has $u(K)=3$, $c_4(K)=2$ and $\sigma(K)=-2$ \cite{Owens2016Immersed}, Theorem~\ref{thm:main} is sufficient to calculate the unknotting number and $4$-ball crossing number of every special alternating knot with at most $10$ crossings.

\subsection{Special alternating knots with crossing number 11}
 Of the 57 prime special alternating knots with crossing number 11, there are 16 for which the unknotting number is currently unknown \cite{knotinfo}:
 \begin{align}\begin{split}\label{list1}
&11a291,\, 11a298,\, 11a299,\, 11a319,\, 11a320,\, 11a329,\, 11a336,\, 11a340,\\
&11a353,\, 11a354,\, 11a356,\, 11a357,\, 11a361,\, 11a362,\, 11a363,\, 11a366.
\end{split}\end{align}
The results of this paper are sufficient to calculate the unknotting number for 11 of these knots. For all 16 of the knots in \eqref{list1} one can verify that they cannot be unknotted by $|\sigma(K)|/2$ crossing changes in an alternating diagram. Thus, by Theorem~\ref{thm:main}, each of these knots satisfies
\[
u(K)\ge c_4(K)\ge \frac{|\sigma(K)|}{2}+1.
\]
 On the other hand, the following 11 knots can be unknotted by $\frac{|\sigma(K)|}{2} +1 $ crossing changes 
 \begin{align}\begin{split}\label{list2}
&11a291,\, 11a298,\, 11a299,\, 11a319,\, 11a320,\, 11a329,\\
&11a336,\, 11a340,\, 11a353,\, 11a356,\, 11a357.
\end{split}\end{align}
 Thus the equality $u(K)=c_4(K)=\frac{|\sigma(K)|}{2}+1$ holds for the knots in \eqref{list2}.

Of the remaining five knots, $11a354$, $11a361$ and $11a366$ all have $\sigma(K)=-4$ and are related to $9_{35}$ by a single crossing change. Since $c_4(9_{35})=2$, this shows $11a354$, $11a361$ and $11a366$ satisfy $c_4(K)=3$.

Table~\ref{tab:newvalues} summarizes the results of these calculations, where $g$ denotes the Seifert genus.

\begin{table}[ht]
\vspace{.2cm}
    \centering
    \begin{tabular}{|c|Y|Y|Y|Y|}
        \hline
        $K$ & $u$ & $c_4$ & $\sigma$ & $g$ \\ \hline
        $11a291$ & 4 & 4 & -6 & 3\\
        $11a298$ & 4 & 4 & -6 & 3\\
        $11a299$ & 3 & 3 & -4 & 2 \\
        $11a319$ & 4 & 4 & -6 & 3\\
        $11a320$ & 3 & 3 & -4 & 2\\
        $11a329$ & 3 & 3 & -4 & 2\\
        $11a336$ & 4 & 4 & -6 & 3\\
        $11a340$ & 4 & 4 & -6 & 3\\
        $11a353$ & 4 & 4 & -6 & 3\\
        $11a354$ & $\{3,4\}$ & 3 & -4 & 2\\
        $11a356$ & 4 & 4 & -6 & 3\\
        $11a357$ & 4 & 4 & -6 & 3\\
        $11a361$ & $\{3,4\}$ & 3 & -4 & 2\\
        $11a362$ & $\{2,3\}$ & $\{2,3\}$ & -2 & 1\\
        $11a363$ & $\{2,3\}$ & $\{2,3\}$ & -2 & 1\\
        $11a366$ & $\{3,4\}$ & 3 & -4 & 2\\ \hline
    \end{tabular}
    \vspace{.4cm}
    \caption{Special alternating knots with crossing number $11$ and previously or currently unknown unknotting number.}
    \label{tab:newvalues}
\end{table}
\vspace{-.7cm}

\subsection{Special alternating knots with crossing number 12}
We repeat the same analysis for special alternating knots with crossing number $12$.
Among the prime special alternating $12$-crossing knots, there are $35$ for which the unknotting number is currently unknown \cite{knotinfo}:
\begin{align}\begin{split}\label{eq:unknown12}
&12a94,\, 12a97,\, 12a102,\, 12a107,\, 12a144,\, 12a145,\, 12a152,\\
&12a156,\, 12a319,\, 12a320,\, 12a368,\, 12a391,\, 12a392,\, 12a421,\\
&12a431,\, 12a443,\, 12a586,\, 12a610,\, 12a653,\, 12a659,\, 12a814,\\
&12a828,\, 12a877,\, 12a880,\, 12a900,\, 12a973,\, 12a974,\, 12a995,\\
&12a996,\, 12a1004,\, 12a1035,\, 12a1037,\, 12a1097,\, 12a1112,\, 12a1113.
\end{split}\end{align}
As in the crossing number $11$ case, we may verify that none of the knots in \eqref{eq:unknown12} can be unknotted in an alternating diagram by $\lvert\sigma(K)\rvert/2$ crossing changes.
Therefore, Theorem~\ref{thm:main} implies that for each of these knots we have
\[
u(K)\ge c_4(K)\ge \frac{\lvert\sigma(K)\rvert}{2}+1.
\]
On the other hand, the following $30$ knots can be unknotted by $\frac{|\sigma(K)|}{2}+1$ crossing changes:
\begin{align}\begin{split}\label{list12}
&12a94,\, 12a97,\, 12a102,\, 12a107,\, 12a144,\, 12a145,\, 12a152,\, 12a319,\\
&12a320,\, 12a368,\, 12a391,\, 12a421,\, 12a431,\, 12a443,\, 12a586,\, 12a610,\\
&12a653,\, 12a659,\, 12a814,\, 12a828,\, 12a877,\, 12a880,\, 12a900,\, 12a973,\\
&12a974,\, 12a995,\, 12a996,\, 12a1004,\, 12a1035,\, 12a1112.
\end{split}\end{align}
Thus the equality $u(K)=c_4(K)=\frac{|\sigma(K)|}{2}+1$ holds for the knots in \eqref{list12}.

For the remaining five knots, the knots $12a156$ and $12a392$ have $\sigma(K)=-4$ and are related to $9_{35}$ by a single crossing change. Again, since $c_4(9_{35})=2$, it follows that $12a156$ and $12a392$ satisfy $c_4(K)=3$.

Table~\ref{tab:newvalues12} summarizes the results of these calculations.

\begin{table}[ht]
\vspace{.2cm}
\centering
\begin{tabular}{|c|Y|Y|Y|Y|}
\hline
\(K\) & \(u\) & \(c_4\) & \(\sigma\) & \(g\) \\ \hline
\(12a94\)   & 4 & 4 & -6 & 3\\
\(12a97\)   & 3 & 3 & -4 & 2\\
\(12a102\)  & 4 & 4 & -6 & 3\\
\(12a107\)  & 4 & 4 & -6 & 3\\
\(12a144\)  & 4 & 4 & -6 & 3\\
\(12a145\)  & 4 & 4 & -6 & 3\\
\(12a152\)  & 3 & 3 & -4 & 2\\
\(12a156\)  & \(\{3,4\}\) & 3 & -4 & 2\\
\(12a319\)  & 4 & 4 & -6 & 3\\
\(12a320\)  & 3 & 3 & -4 & 2\\
\(12a368\)  & 4 & 4 & -6 & 3\\
\(12a391\)  & 4 & 4 & -6 & 3\\
\(12a392\)  & \(\{3,4\}\) & 3 & -4 & 2\\
\(12a421\)  & 3 & 3 & -4 & 2\\
\(12a431\)  & 4 & 4 & -6 & 3\\
\(12a443\)  & 3 & 3 & -4 & 2\\
\(12a586\)  & 4 & 4 & -6 & 3\\
\(12a610\)  & 3 & 3 & -4 & 2\\
\(12a653\)  & 3 & 3 & -4 & 2\\
\(12a659\)  & 4 & 4 & -6 & 3\\
\(12a814\)  & 4 & 4 & -6 & 3\\
\(12a828\)  & 4 & 4 & -6 & 3\\
\(12a877\)  & 4 & 4 & -6 & 3\\
\(12a880\)  & 3 & 3 & -4 & 2\\
\(12a900\)  & 4 & 4 & -6 & 3\\
\(12a973\)  & 4 & 4 & -6 & 3\\
\(12a974\)  & 3 & 3 & -4 & 2\\
\(12a995\)  & 4 & 4 & -6 & 3\\
\(12a996\)  & 3 & 3 & -4 & 2\\
\(12a1004\) & 4 & 4 & -6 & 3\\
\(12a1035\) & 4 & 4 & -6 & 3\\
\(12a1037\) & \(\{3,4\}\) & \(\{3,4\}\) & -4 & 2\\
\(12a1097\) & \(\{3,4\}\) & \(\{3,4\}\) & -4 & 2\\
\(12a1112\) & 4 & 4 & -6 & 3\\
\(12a1113\) & \(\{3,4\}\) & \(\{3,4\}\) & -4 & 2\\ \hline
\end{tabular}
\vspace{.4cm}
\caption{Special alternating knots with crossing number \(12\) and previously or currently unknown unknotting number.}
\label{tab:newvalues12}
\end{table}
\vspace{-.7cm}

%\newpage
\let\MRhref\undefined
\bibliographystyle{hamsalpha}
\bibliography{bib}

\end{document}